\documentclass[11pt]{amsart}
\usepackage{amsmath,amsthm,amscd,amssymb, color}
\usepackage{latexsym}
\usepackage[colorlinks, citecolor = blue]{hyperref}
\usepackage[capitalize, nameinlink]{cleveref}
\usepackage[alphabetic]{amsrefs}
\usepackage{pgfplots}
\usepackage[font=footnotesize]{caption}

\newif\ifgraphs
\graphstrue
\graphsfalse 

\usepackage{enumerate}

\usepackage[margin=1.4in]{geometry}

\usepackage{pgfplots}
\usepgfplotslibrary{external} 

\newcommand{\R}{\mathbb R}
\newcommand{\Q}{\mathbb Q}

\newcommand{\Z}{\mathbb Z}
\newcommand{\F}{\mathbb F}
\newcommand{\eps}{\varepsilon}

\newcommand{\calA}{\mathcal A}
\newcommand{\calE}{\mathcal E}
\newcommand{\Epr}{\calE_{\mathrm{pr}}}
\newcommand{\Eall}{\calE_{\mathrm{all}}}
\newcommand{\Eht}{\calE_{\mathrm{ht}}}

\newcommand{\new}{\mathrm{new}}

\newcommand{\bmx}{\left( \begin{matrix}}
\newcommand{\emx}{\end{matrix} \right)}

\renewcommand{\mod}{\bmod}
\newcommand{\leg}{\overwithdelims ()}

\DeclareMathOperator{\tr}{tr} 
\DeclareMathOperator{\GL}{GL}

\newtheorem{lem}{Lemma}
\numberwithin{lem}{section}
\newtheorem{prop}[lem]{Proposition}

\newtheorem{cor}[lem]{Corollary}
\newtheorem{conj}[lem]{Conjecture}

\numberwithin{rem}{section}

\crefname{conj}{Conjecture}{Conjectures}

\numberwithin{equation}{section}

\title{Rank bias for elliptic curves mod $p$}

\author{Kimball Martin}
\author{Thomas Pharis}
\address{Department of Mathematics, University of Oklahoma, Norman, OK 73019}
\email{kimball.martin@ou.edu}
\address{Department of Mathematics, Indiana University, Bloomington, IN 47405}
\date{\today}

\begin{document}

\maketitle

\begin{abstract} 
We conjecture that, for a fixed prime $p$, rational elliptic curves
with higher rank tend to have more points mod $p$.  We show that there is
an analogous bias for modular forms with respect to root numbers, 
and conjecture that the order of the rank bias
for elliptic curves is greater than that of the root number bias for modular forms.
\end{abstract}

\section{Introduction and conjectures} \label{sec1}

A natural question is whether rational elliptic curves $E$ with more global points
have more points mod $p$.  More precisely, does 
$a_p = a_p(E) =  p+1 - \# E(\mathbb F_p)$ (for $p$ a good prime) 
tend to be smaller when $E$ has larger rank?

Much is understood in the ``horizontal'' direction.
Namely if $E$ is fixed, then the Sato--Tate conjecture (now a theorem by \cite{BGHT}) 
asserts the limiting distribution
of $\frac{a_p}{\sqrt p}$ is independent of the rank of $E$.  However, if one numerically
computes examples for various $E$, one notices an apparent bias---the 
$a_p$'s tend to be smaller when $E$ has larger rank.  Necessarily such a bias would be 
bounded by the error term in the convergence to the Sato--Tate distribution.  
Indeed, one may interpret the Birch and Swinnerton-Dyer conjecture 
as a certain measure of this bias.  
In its original formulation in \cite{BSD}, it asserts that, up to a constant,
the order of growth of $\prod_{p < X} \frac{\# E(\mathbb F_p)}{p}$ is $(\log X)^r$,
where $p$ runs over primes of good reduction and $r$ is the rank of $E(\Q)$.
This may be very loosely interpreted as saying that elliptic curves with higher ranks
have more points mod $p$ for large primes $p$.
We also remark that Nagao's conjecture makes a similar prediction for averages
over 1-parameter families of elliptic curves.

Here we investigate the above question in the ``vertical'' direction: for a fixed prime
$p$ and varying elliptic curves $E$ of rank $r$, does $a_p$ tend to be smaller
the larger $r$ is?  Or to put it loosely: do elliptic curves with higher ranks also have
more points mod $p$ for small primes $p$?
A basic issue is how to try to measure such a bias,
as the Birch and Swinnerton-Dyer framework (essentially a weighted geometric mean of
$a_p$'s as $p$ varies) has no obvious analogue for a fixed $p$.

Since such problems seem very difficult to tackle theoretically, we investigate
this question computationally, and theoretically study an analogous question
for modular forms.  We first discuss the case of modular forms, which will help
motivate our framework for measuring and conjecturing bias for elliptic curves.

\subsection{Root number bias for modular forms}
Elliptic curves  $E$ of conductor $N$ correspond to rational newforms $f \in S_2(N)$
of weight 2 and level $N$ such that $L(s,E) = L(s,f)$, and thus each $a_p(f) = a_p(E)$.  
The root number $w$ 
of $f$ (or $E$) is $\pm 1$, which is the sign in the functional equation of the $L$-function.  
The analytic rank of this $L$-function is even or odd according to whether
$w$ is $+1$ or $-1$.  According to the Birch and Swinnerton-Dyer conjecture,
the analytic rank should be the same as the algebraic rank.  
The minimalist conjecture predicts that 100\% of the time, the rank is 0 or 1,
according to the root number being $+1$ or $-1$.  

In \cref{cor:mf-bias}, we show that for a fixed prime $p$, the average of $a_p(f)$ 
over newforms $f \in S_2(N)$ with root number $\pm 1$ grows approximately like
$\pm\frac {1}{\sqrt N}$ for large squarefree $N$ prime to $p$.
In fact, we more generally treat newforms of even weight $k$,
and the method applies to arbitrary Fourier coefficients $a_n(f)$, however the
signs for $k \equiv 0 \mod 4$ are opposite to those for $k \equiv 2 \mod 4$.

We now recast this bias in terms of weighted averages of Fourier coefficients over
varying $N$.
Let $\mathcal F^\pm(X)$ be the union of the sets of newforms $f$ with root number
$\pm 1$ in $S_2(N)$, as $N$ ranges over squarefree levels less than $X$.
For a newform $f$, let $N_f$ denote its exact level, i.e., $f$ is a newform in $S_2(N_f)$.
Let $\phi: \mathbb N \to \R_{>0}$ be a monotonic non-decreasing function
of at most polynomial growth.  We call such a $\phi$ a \emph{weight function}.
Consider the weighted average of $p$-th Fourier coefficients,
\begin{equation} \label{eq:mf-avg}
 \calA^\pm(p, X; \phi) = \frac 1{\# \mathcal F^\pm(X)}\sum_{f \in \mathcal F^\pm(X)} a_p(f) \phi(N_f).
\end{equation}
Then \cref{cor:mf-bias} implies the following.

\begin{prop} \label{prop1}
Fix a prime $p$, a weight function $\phi$ as above, and $\epsilon > 0$.
\begin{enumerate}
\item For any $\phi$, $\mathcal A^+(p, X; \phi) > \mathcal A^-(p, X; \phi)$ for sufficiently
large $X$.

\item If $\phi(N) \ll N^{\frac 12 - \epsilon}$, then $\mathcal A^\pm(p, X; \phi) \to 0$ as $X \to \infty$.

\item If $\phi(N) \gg N^{\frac 12 + \epsilon}$, then $\mathcal A^\pm(p, X; \phi) \to \pm \infty$
as $X \to \infty$.
\end{enumerate}
\end{prop}

The polynomial growth condition on $\phi$ is not actually needed for this proposition,
but we impose it for the purposes of preventing erratic behavior in
weighted averages for sequences with greater variation.

The first statement asserts that there is a persistent root number bias in the $a_p$'s;
in fact, for any $\phi$, $\mathcal A^+(p, X; \phi) > 0$ and 
$\mathcal A^-(p, X; \phi) < 0$ for sufficiently large $X$, which corresponds to the
sign matching in assertion (3).
The latter two statements quantify the size of the bias---they 
say the bias is roughly on the order of the inverse square root of the level.
While \cref{prop1} is less precise than \cref{cor:mf-bias},
this formulation provides a model for investigating and measuring bias for elliptic curves. 

Our original question about the rank elliptic curves affecting the size of
$a_p$'s is geometrically motivated.  One can also interpret the bias
in \cref{prop1} geometrically as follows.

Let $f \in S_2(N)$ be a newform with rationality field
$K_f$.  Its rationality degree $d = [K_f: \Q]$ equals the number of Galois conjugates 
$f^\sigma$ of $f$.  Now (the Galois orbit of) $f$ corresponds to a rational 
abelian variety $A_f$ of dimension $d$ such that $L(s,A_f) = \prod_\sigma L(s,f^\sigma)$.
Let $a_p(A_f) = \sum_\sigma a_p(f^\sigma) = \tr_{K_f/\Q} a_p(f)$.  Then, for $p \nmid N$,
$\#A_f(\F_p) = \prod_\sigma (p+1 - a_p(f^\sigma))$,
which is a degree $d$ monic polynomial in $p$, and the coefficient of $p^{d-1}$
is $-a_p(A_f)$.  Hence we may think of 
$a_p(A_f) p^{d - 1}$ being a ``first order estimate'' for $p^d - \#A_f(\F_p)$.

For simplicity, say $N$ is prime.  Then $S_2(N)$ has $2$ Atkin--Lehner
eigenspaces, each of size approximately $\frac{N-1}{24}$ by \cite{me:ref-dim}, 
corresponding to root numbers $+1$ and $-1$.  
We conjectured in \cite{me:maeda} that each Atkin--Lehner space
is generated by a single Galois orbit 100\% of the time.  Suppose
the root number $+1$ (resp.\ $-1$) newforms are all Galois conjugates of 
a single newform $f_+$ (resp.\ $f_-$).
The new part of the Jacobian $J_0(N)$ of $X_0(N)$ decomposes into 2 simple 
pieces $A_+ \oplus A_-$, where $A_\pm = A_{f_\pm}$, and
the average of the $a_p(f)$'s with root number $\pm 1$ is
approximately $\frac{24}{N-1} \cdot a_p(A_{\pm})$.  By the minimalist
philosophy, we further expect that 100\% of the time $A_+$ has rank 0 
and $A_-$ has rank $\dim A_- \approx \frac{N-1}{24}$.  
Now \cref{prop1} (or rather \cref{cor:mf-bias} to restrict to prime levels) 
says that, for fixed $p \nmid N$ and prime $N \to \infty$, 
$a_p(A_\pm)$ grows roughly like $\pm \sqrt N$.  In particular,
the $a_p(A)$'s tend to be smaller for higher-rank modular (GL(2) type) 
abelian varieties along this family.
These remarks extend to the case of squarefree $N$ with suitable modification.

\subsection{Rank bias for elliptic curves}
Two common ways of counting rational elliptic curves $E$ are to (partially) order
by conductor or (a suitable choice of) height. 
In addition, one can either count isomorphism classes or isogeny classes of curves, 
but it seems likely that this distinction will not significantly affect the statistics we consider.
Recall that $a_p(E)$ only depends on the isogeny class of $E$.

For definiteness, we let $\mathcal E$ be one of the following three families of 
(partially ordered classes of) rational elliptic curves: 
(i) isogeny classes of elliptic curves of prime conductor, (partially) ordered by conductor;
(ii) isogeny classes of all elliptic curves, ordered by conductor;
and (iii) isomorphism classes of all elliptic curves, 
ordered by (minimal) naive height.
These families are respectively denoted by $\Epr$, $\Eall$, and $\Eht$.

When $\calE = \Epr$ or $\Eall$, let $|E| = N_E$ denote the conductor of $E$.
When $\calE = \Eht$, let $|E|=H_E$ denote the minimum naive height in the isomorphism
class of $E$.
Let $\mathcal E(X)$ be the set of (classes of) elliptic curves in $\mathcal E$
with $|E| < X$. 
Let $\mathcal E_r(X)$ be the subset of $E$ in $\mathcal E(X)$ of rank $r$.
Set $\mathcal E_r = \bigcup_X \mathcal E_r(X)$.

For a weight function $\phi$ as above, we define the weighted average
\begin{equation} \label{eq:wt-avg}
 \calA_r(p, X; \phi) = \calA^{\mathcal E}_r(p, X; \phi) = \frac 1{\# \mathcal E_r(X)}\sum_{E \in \mathcal E_r(X), \, |E| < X, \, (N_E,p) = 1} a_p(E) \phi(N_E).
\end{equation}
Note that in all cases we are weighting by conductor, in analogy with the weighting
of modular forms by level in \eqref{eq:mf-avg}.
When ordering by height $H_E$, one could instead weight the averages 
by $\phi(H_E)$ rather than $\phi(N_E)$, and we briefly remark on this below.

The following two conjectures assert an analogue of \cref{prop1}
for quantifying a rank bias in the $a_p$'s of elliptic curves.

\begin{conj} [Existence of rank bias] \label{conj1}
Let $\mathcal E$ be $\Epr$, $\Eall$ or $\Eht$,
and fix $r \ge 0$.   Suppose $\mathcal E_r$ and $\mathcal E_{r+1}$ are infinite.
Then for any weight function $\phi$, 
$\calA_r(p, X; \phi) > \calA_{r+1}(p, X; \phi)$ for all sufficiently large $X$.
\end{conj}

\begin{conj}  [Order of rank bias] \label{conj2}
Let $\mathcal E$ be $\Epr$, $\Eall$ or $\Eht$, and fix $r \ge 0$. 
Suppose $\mathcal E_r$ is infinite.   
Then there exists a $\delta > 0$ 
such that for any $\epsilon > 0$ the following hold.

\begin{enumerate}
\item If $\phi(N) \ll (\log N)^{\delta - \epsilon}$, then $\calA_r(p, X; \phi) \to 0$ as $X \to \infty$

\item If $\phi(N) \gg (\log N)^{\delta - \epsilon}$, then $|\calA_r(p, X; \phi)| \to \infty$.  Moreover $\calA_r(p, X; \phi) \to +\infty$ if $r \le 1$ and
$\calA_r(p, X; \phi) \to -\infty$ if $r \ge 2$.
\end{enumerate}
\end{conj}

The first conjecture asserts that there is a persistent bias toward the
$a_p$'s being smaller for larger ranks, when we consider sufficiently large
families of elliptic curves.  Note that in special families, e.g., quadratic twists,
$a_p$'s can be have in special ways which do not exhibit such a bias.
It is not known for which ranks $r$ there are infinitely many, or even any, 
curves of rank $r$ (see \cite{PPVW} for recent conjectures).
Hence we impose hypotheses on the infinitude of $\calE_r$ and $\calE_{r+1}$ 
to try to avoid considering families which are too small or special.

The second conjecture asserts that the rank bias has roughly inverse
polylogarithmic order in the conductor.  In fact, it is plausible that the order of 
bias is simply inverse logarithmic, i.e.,  one can take $\delta = 1$ 
in \cref{conj2}, but the data are not entirely clear (see \cref{sec3}).  
Note that if we take the
constant weight function $\phi(N) = 1$, \cref{conj2}(1) asserts
that the honest averages of the $a_p$'s
tend to $0$, similar to the case of modular forms in \cref{prop1}.  However,
there are a couple of obvious differences from the situation of \cref{prop1}.

One evident difference is that the approximate order of this bias ($(\log N)^{-\delta}$) is larger than what we saw for modular forms ($N^{-\frac 12}$).  
From the geometric interpretation of \cref{prop1}, there is no
obvious guess for the order of bias of $a_p(A)$'s for modular abelian varieties
as a simultaneous function of rank, dimension and conductor, 
since the dimension, conductor and
rank are all strongly correlated that context.  However, the notion that
the order of rank bias is larger than any $N^{-\epsilon}$ when we restrict to 
$A$ of bounded dimension, e.g., for elliptic curves, at least seems compatible
with \cref{prop1}.

Another difference is that we are measuring rank bias for elliptic curves,
but root number bias for modular forms.  By the minimalist conjecture, it is natural
to expect that if we look at averages of $a_p(E)$'s where $E$ has root number
$+1$ (resp.\ $-1$), this should behave the same as the rank 0 (resp.\ rank 1) averages.
However, \cref{conj2}(2) says that the $a_p(E)$'s tend to be positive for both rank 0 and
rank 1, rather than matching the sign of the root number!

It could happen that, say, the contribution from the rank 3 and rank 5 elliptic curves is significant enough to make the (suitably weighted averages of) 
$a_p(E)$'s tend to be negative for root number $-1$.
On the other hand, it also seems theoretically plausible
that for either root number $+1$ or $-1$ (and thus for all ranks combined), 
the $a_p(E)$'s tend to be positive,
and that the geometric reason for the root number bias in signs for modular forms 
is due to large rank abelian varieties.  Exploratory calculations did not exhibit
a positive bias in (weighted averages of) $a_p(E)$'s with fixed root 
number.  However, it is hard to draw clear conclusions from these calculations
because the numerical convergence of ranks 
to the minimalist conjecture is very slow, i.e., in the range we are able
to compute, there is still a very large proportion of rank $\ge 2$ curves (see below).

We do remark that if we count integral Weierstrass equations of elliptic curves 
ordered by height (which is almost the same as our family $\Eht$), 
then the reductions mod $p$ are evenly distributed.  By a result of Birch \cite{birch},
this means the unweighted averages of $a_p(E)$'s over all ranks for this family 
tends to 0.  This is compatible with \cref{conj2}(1), and it does not seem to 
preclude the possibility that suitably weighted averages of $a_p(E)$'s over all
ranks may be positive.

\subsection{Evidence and meta-analysis}

Our evidence for these conjectures is purely computational, and is presented 
in \cref{sec3}.

For $\Epr$ and $\Eall$,
we estimate weighted averages using 
the Stein--Watkins databases \cite{stein-watkins} consisting of over 11 milion
isogeny classes of prime conductor $N < 10^{10}$ and over 115 million isogeny
classes of arbitrary conductors $N \le 10^8$.  The Stein--Watkins databases do not
 catalogue all isogeny classes in these conductor ranges, but at least the
Stein--Watkins prime conductor database appears to be nearly complete:
\cite{BGR} estimates it contains over 99.8\% of curves with prime conductor $N < 10^{10}$.
(In fact \cite{BGR} computed a much larger database of prime conductor elliptic curves,
but that database does not include rank calculations which we require.)
For $\Eht$, we compute weighted averages using the height database 
from \cite{height}, which contains all of 
the over 238 million curves with naive height $H \le 2.7 \cdot 10^{10}$.

In fact, the reason for formulating our conjectures for the  three specific
families $\Epr$, $\Eall$ and $\Eht$ is that they correspond to these existing
extensive databases of curves that include conjectural ranks.  
The Stein--Watkins databases include numerically computed analytic ranks.
The ranks computed in the height database in general assume several standard 
conjectures, but are unconditional over 80\% of the time.

For these 3 databases, we compute weighted averages of $a_p(E)$'s with 
$|E| < X$ and fixed rank $r \le 5$ for both a variety of weight functions $\phi$,
and a variety of primes $p \le 300$.  Note that there are not enough curves in these
databases to get meaningful statistics for $r \ge 6$.
For a fixed weight function $\phi$ and rank $r$, we found the general 
behavior to be more-or-less similar for each of the 3 databases and 
for any choice of $p$.  However we will point out a couple of apparent
exceptions to this trend in \cref{sec3}.

When $\phi(N) = 1$, the weighted average
graphs quite quickly tend to zero as asserted in \cref{conj2}(1), 
and do not appear to cross each other for different ranks beyond very small $X$,
as asserted in \cref{conj1}.
When $\phi(N) \gg \sqrt{N}$ or larger, the weighted averages also clearly tend to
$\pm \infty$.  More generally, for $\phi(N) = N^\delta$, the weighted average graphs
appear to grow like $C X^\delta (1 + O(X^\epsilon))$ for some constant nonzero $C$.
For $\phi(N) = \log(N)$, some graphs appear as though they may 
have a finite nonzero asymptote, and some as though they may be slowly
increasing or decreasing, depending on both $p$ and $r$.  
For $\phi(N) = \log \log N$ most graphs appear to go to 0,
and for $\phi(N) = (\log N)^2$ most graphs appear to go to $\pm \infty$ slowly. 
In summary, the data support \cref{conj1,conj2} quite well, possibly with
$\delta = 1$ in \cref{conj2}.  It may be that
in order for $\calA_r(p, X; \phi)$ to tend to a nonzero constant one needs
to make a more complicated choice of $\phi(N)$, which potentially 
depends on $r$.

We expect these conjectures are fairly robust with respect to the choice of family.
E.g., the conjectures should be unchanged if one looks at isomorphism rather than isogeny classes of elliptic curves ordered by conductor (the Stein--Watkins databases
also include isomorphism classes), or if one restricts to squarefree conductors.
We have also computed some weighted averages in the family $\Eht$ where one weights
by $\phi(H_E)$ rather than $\phi(N_E)$, and the general behavior appears similar.
For brevity, we have not included details of weighting by height.

It is well known that one needs to compute very far out to get 
convincing emperical evidence for the minimalist conjecture, i.e., that the average
rank of elliptic curves tends to 0.5.  Indeed, the average rank is numerically increasing
in the Stein--Watkins database for general conductors---see \cite{BMSW}.
However, in the Stein--Watkins prime conductor database, we see that the average rank 
per isogeny class quickly goes up to just over 0.98, and then gradually decreases to approximately 0.96544.  Moreover, in \cite{height}, the authors find that the average rank increases to about 0.908 around height $6 \cdot 10^8$, and then decreases to around 0.901
by height $2.7 \cdot 10^{10}$.

Given this, it is natural to wonder how much one can trust that our calculations
are representative of asymptotic behavior.  First, since we are separating by rank,
there is no direct effect of the slow convergence to the minimalist conjecture on our
data (except that it means we have many curves of rank $\ge 2$ in our databases,
which is actually helpful for our experiments).  
It is of course possible that some of our graphs which appear to have a nonzero (or infinite)
limit, which would signify a persistent bias, eventually tend to 0, or vice versa.
However, we find the asymptotics of the graphs quite compelling up to a factor of
order $(1+O(X^\epsilon))$.

We also note that many statistics besides average ranks converge to expectations 
rather quickly.  For instance, in the horizontal direction,
the $a_p$'s tend to the Sato--Tate distributions quite quickly.
In the vertical direction, numerical convergence to the parity principle---the 
notion that half of all curves should have root number $+1$ and half have $-1$---is 
also quite fast (e.g., see \cite{BMSW}).

Consequently we find the computational evidence very convincing of the existence of
a rank bias on the order of $O(N^{-\eps})$, 
even if it is difficult to tease out the exact order of bias from our calculations.

\section{Traces of Hecke operators} \label{sec2}

Here we exhibit a bias in the traces of Hecke operators $T_n$ on spaces of newforms 
with fixed root numbers.  Note that in \cite[Section 2]{me:ref-dim},
explicit dimension formulas for these spaces were proven, and a strict bias 
towards root number $+1$ was exhibited.  That may be viewed as the
$n=1$ analogue of what we do here.

First we set our notation.  
Denote by $\omega(N)$ the number of prime divisors of $N$, by $\sigma_1(n)$
the sum-of-divisors function, and by
$\delta_{i,j}$ the Kronecker $\delta$ function.
Let $H(n)$ be the Hurwitz class number, i.e., the number of
$\mathrm{SL}_2(\Z)$-equivalence classes of positive definite integral
binary quadratic forms $Q$ of discriminant $-n$ weighted by 
$\# \mathrm{Aut}(Q)/2$.
Throughout, multiple occurrences of $\pm$ and/or $\mp$ within a single statement
are to be interpreted as cases dependent on the first occurrence.

Let $S_k(N)$ denote the space of holomorphic even weight $k$ cusp forms for 
$\Gamma_0(N)$, and
$S^\new_k(N)$ the subspace spanned by newforms.  For $S = S_k(N)$ or 
$S = S^\new_k(N)$, denote by $S^\pm$ the subspace of $S$ spanned by
eigenforms with root number $\pm 1$.  Let $W_N = \prod_{p | N} W_p$, where
$W_p$ denotes the $p$-th Atkin--Lehner operator on $S_k(N)$.
For a subspace $S$ of $S_k(N)$ and an operator $T$ on $S_k(N)$ which leaves
$S$ invariant, denote by $\tr_S T$ the trace of the restriction of $T$ to $S$.

\begin{lem} Suppose $N$ is squarefree, and $(n,N) = 1$.  Then
$\tr_{S_k^\new(N)} T_n W_N = \tr_{S_k(N)} T_n W_N$.
\end{lem}

\begin{proof} This is a special case of \cite[Proposition 2]{yamauchi}.  Here is
an alternative argument in terms of representations, which we find more enlightening.

Suppose $f \in S^\new_k(M)$ is a newform, where $M$ is a proper divisor of $N$.
Let $\pi = \bigotimes \pi_v$ be the associated cuspidal representation of $\GL(2)$,
and $\phi = \bigotimes \phi_v$ be a newvector associated to $f$.  Then
the contribution $\pi^{K(N)}$ of $\pi$ to $S_k(N)$ has a basis of the form $W_Q \phi$
where $Q$ ranges over positive divisors of $M$, and $W_Q \phi = \bigotimes \phi_v'$
where $\phi_q' = \pi_q \bmx & 1 \\ q & \emx \phi_q$ if $q | Q$ and $\pi_v' = \pi_v$ if
$v \nmid Q$.  Now observe that when $(n,N) = 1$, $T_n$ acts by a scalar on 
$\pi^{K(N)}$ and $W_N$
acts as an involution on the above basis elements with no fixed points.  Thus
$T_n W_N$ has trace zero on $\pi^{K(N)}$, and consequently on the whole old space of
$S_k(N)$.
\end{proof}

\begin{prop} Suppose $N$ is squarefree, $n > 1$ is nonsquare, $(n,N) = 1$ and $N > 4n$.  
Then
\[ \left| \tr_{S_k^\new(N)^\pm} T_n \mp \frac 14 (-n)^{\frac{k-2}2} H(4nN) \right| < 
\left( 2^{\omega(N)} (4n)^{\frac k2} + \delta_{k,2} \right)\sigma_1(n). \]
\end{prop}

We remark the restriction to $n$ being nonsquare here is purely for simplicity.
Otherwise there is an extra term in the explicit formula for $\tr_{S_k^\new(N)} T_n$
arising in the proof below.

\begin{proof} Let $n$ be a positive integer coprime to $N$.
Since the root number of a newform in $S_k(N)$ is $(-1)^{\frac k2}$ times $W_N$,
we have
\begin{equation} \label{eq:tr-id}
 \tr_{S_k^\new(N)^\pm} T_n = \frac 12 
\left( \tr_{S_k^\new(N)} T_n  \pm (-1)^{\frac k2} \tr_{S_k^\new(N)} T_n W_N \right).
\end{equation}

Yamauchi \cite{yamauchi} proved a formula for 
$\tr_{S_k(N)} T_n W_N$ for general $N$, though that work contained clerical errors.
A corrected form was given by Skoruppa and Zagier \cite[(2.7)]{skoruppa-zagier},
which for squarefree $N$ simplifies to:
\[ \tr_{S_k(N)} T_n W_N = - \frac 12 \sum_{s^2 \le 4nN, N | s}
p_{k}(s/\sqrt N, n) H(4nN - s^2) + \delta_{k,2} \sigma_1(n). \]
Here, when $b^2-4c \ne 0$, 
$p_k(b,c) =(\rho_1^{k-1} - \rho_2^{k-1})/(\rho_1 - \rho_2)$ where $\rho_1, \rho_2$
are the roots of $x^2 - bx + c$.
If $N > 4n$, we only get the $s=0$ term in the first sum:
\begin{equation} \label{eq:SZ}
 \tr_{S_k(N)} T_n W_N = - \frac 12 n^{\frac{k-2}2} H(4nN) + \delta_{k,2} \sigma_1(n).
\end{equation}
By the above lemma, we now have an explicit formula for 
$\tr_{S_k^\new(N)} T_n W_N$.

An explicit formula for $\tr_{S_k^\new(N)} T_n$ is given in 
\cite[Theorem 5]{murty-sinha} for arbitrary $N$.  When $N$ is squarefree and
$n > 1$ is nonsquare, this gives
\[  \tr_{S_k^\new(N)} T_n  = - \frac 12 \sum_{t^2 < 4n}
p_k(t,n) \sum_{f^2 | (4n-t^2)} h_w(\frac{t^2-4n}{f^2}) B_2(N)_f + 
 \delta_{k,2} \mu(N) \sigma_1(n). \]
Here $t \in \Z$, $f \in \mathbb N$, $h_w(D)$ is the class number of the
imaginary quadratic order $\mathcal O(D)$ of discriminant $D$ times 
$[\mathcal O(D)^\times : \Z^\times ]$ (interpreted as $0$ if $D$ is not a discriminant), 
and $B_2(N)_f = \prod_{p | N} B_2(p)_f$
where $B_2(p)_f$ is $p-1$ if $p | f$ and ${t^2 - 4n \leg p} - 1$ otherwise.

We note that $\sum_{f^2 | (4n-t^2)} h_w(\frac{t^2-4n}{f^2}) = H(4n-t^2)$ and 
and $\sum_{t^2 < 4n} H(4n-t^2) < 2\sigma_1(n) - 1$
(e.g., see \cite[Proposition 12]{murty-sinha}).
Since also $|p_k(t,n)| < 2 (4n)^{\frac {k-1}2}$ and
$B_2(N)_f \le f 2^{\omega(N)}$, we have
\begin{equation} \label{eq:MS}
 \left| \tr_{S_k^\new(N)} T_n \right| < \left(2^{\omega(N)+1}(4n)^{\frac k2} + \delta_{k,2}\right) \sigma_1(n).
\end{equation}
(See also \cite[Proposition 14]{murty-sinha} for a similar bound.)

Now combine \eqref{eq:tr-id} with \eqref{eq:SZ} and \eqref{eq:MS}.
\end{proof}

In fact one can use the formulas in \cite{yamauchi}, \cite{skoruppa-zagier} and 
\cite{murty-sinha}
to give explicit formulas for $\tr_{S_k^\new(N)^\pm} T_n$ without assuming
$N$ is squarefree (but still coprime to $n$) or larger than $4n$.  However, such formulas 
will involve alternating sums of class numbers, which makes it more difficult to
generalize the following corollary to $N$ non-squarefree.

\begin{cor} \label{cor:mf-bias}
 Fix $k \ge 2$ even, $n > 1$ squarefree, and $\epsilon > 0$.
As $N \to \infty$ along a sequence of
squarefree numbers coprime to $N$, we have
\[ N^{\frac 12 - \epsilon} \ll  \pm (-1)^{\frac{k-2}2}  \tr_{S_k^\new(N)^{\pm}} T_n \ll N^{\frac 12} \log N. \]
\end{cor}

\begin{proof}
Note that, along a sequence of squarefree $N$, $2^{\omega(N)} = O(N^\epsilon)$;
see for instance the proof of \cite[Proposition 3.10(ii)]{me:qmf-zeroes}.
Consequently, 
\[ \tr_{S_k^\new(N)^{\pm}} T_n = \pm \frac 14 (-n)^{\frac {k-2}2} H(4nN)
+ O (N^\epsilon).\]  
Since $nN$ is squarefree, $4nN$ is a fundamental discriminant
and $H(4nN)$ is the usual class number $h(-4nN)$.  
Now for fundamental discriminants $-D < 0$,
we use the standard upper bound $h(-D) \ll D^{\frac 12} \log D$ and Siegel's (ineffective) lower bound $h(-D) \gg D^{\frac 12 - \epsilon}$.
\end{proof}

Roughly, this says that  $\tr_{S_k^\new(N)^{\pm}} T_p$ grows approximately like 
$\pm \sqrt N$ when $k \equiv 2 \mod 4$, and approximately like $\mp \sqrt N$
when $k \equiv 0 \mod 4$.
In particular, when $k=2$ and $\phi$ is any weight function $\phi$, we have that
$\mathcal A^+(p, X; \phi) > 0$ and $\mathcal A^-(p, X; \phi) < 0$ for $X$ sufficiently
large.  This proves the first part of \cref{prop1}.

Moreover, by \cite[Section 2]{me:ref-dim} we have
 $\dim S_2^\new(N)^{\pm} = \frac{\varphi(N)}{12} + O(N^{\frac 12} \log N)$.
 Here $\varphi$ denotes the Euler totient, not to be confused with a weight function $\phi$.
Now note that, for any $\epsilon > 0$,
we have $X^{2-\epsilon} \ll \sum_{N < X} \varphi(N) \ll  X^2$, where in
the sum $N$ is restricted to positive squarefree integers.  For the lower
bound, we are using that the squarefree integers have positive natural density in 
$\mathbb N$ and satisfy $\varphi(N) \ge \frac N{2^{\omega(N)}}$, together with 
the abovementioned  fact that $2^{\omega(N)} = O(N^\epsilon)$.
Consequently, we see that
\[ \frac{\sum_{N < X} N^{\frac 12 - \epsilon} \phi(N)}{X^2}  \ll |\mathcal A^\pm (p, X; \phi)|
\ll \frac{\sum_{N < X} N^{\frac 12} \log N \phi(N)}{X^{2-\epsilon}} \]
for any weight function $\phi$ and constant $\epsilon > 0$.
(In both sums, $N$ is restricted to positive squarefree $N$.)  This immediately gives
the remainder of \cref{prop1}.

\section{Data} \label{sec3}

Now we present and briefly discuss some data supporting our conjectures.

We computed weighted averages as in \eqref{eq:wt-avg} for the finite subfamilies
$\Epr^{\mathrm{SW}}$, $\Eall^{\mathrm{SW}}$ and $\Eht^{\mathrm{db}}$ of
$\Epr$, $\Eall$ and $\Eht$ which respectively consist of all classes of curves
contained within the Stein--Watkins prime conductor database, Stein--Watkins 
arbitrary conductor
database and the height database from \cite{height}.  
These calculations only approximate the averages in \eqref{eq:wt-avg}
for $X < 10^{10}$, $X \le 10^8$ and $X \le 2.7 \times 10^{10}$ for two reasons
mentioned in \cref{sec1}:
the Stein--Watkins databases are incomplete and many of the ranks in the
 databases are conjectural.

Assuming correctness of the ranks, we computed the weighted averages
$\calA_r^\calE(p, X; \phi)$ for $\calE = \Epr^{\mathrm{SW}}$, $\Eall^{\mathrm{SW}}$ and 
$\Eht^{\mathrm{db}}$ for a wide variety of primes $p \le 300$
and weight functions $\phi$ for $X$ up to the relevant database bound.  
On a single CPU core, running calculations for several
$p$ at a time, the calculations for a given $\phi$ took approximately 75--100 minutes
of real time for $\Epr^{\mathrm{SW}}$, 13--16 hours for $\Eall^{\mathrm{SW}}$,
and 24--28 hours for $\Eht^{\mathrm{db}}$.

For a fixed $r \le 4$ and $\phi$, the behavior of $\calA_r^\calE(p, X; \phi)$ is generally
similar for both different $p$ and $\calE$.  See \cref{fig:Ecomp} for overlaid
graphs with $p=7$, $\phi(N) = \log N$, and various $r$ for each $\calE$. 
These graphs strongly support \cref{conj1}, at least for $\phi(N) = \log N$.
We have examined similar graphs for a variety of weight functions $\phi$,
and these graphs are equally convincing in support of \cref{conj1}.
E.g., see \cref{fig:un-sqrt,fig:x-x2} for the case $\calE = \Epr^{\mathrm{SW}}$,
$p = 7$ with the weights 
$\phi(N) = 1$, $\phi(N) = \sqrt N$, $\phi(N) = N$ and $\phi(N) = N^2$.

\begin{figure}
\begin{minipage}{.333\textwidth}
\resizebox{.9\linewidth}{!}{
\begin{tikzpicture}
\begin{axis}[cycle list name=color list, legend pos = south east, ymin=-60]
\addplot table {data/prime_log_r0_p7.dat};
\addplot table {data/prime_log_r1_p7.dat};
\addplot table {data/prime_log_r2_p7.dat};
\addplot table {data/prime_log_r3_p7.dat};
\addplot table {data/prime_log_r4_p7.dat};
\legend{$r= 0$, $r=1$, $r=2$, $r=3$, $r=4$}
\end{axis}
\end{tikzpicture} 
}
\end{minipage}%
\begin{minipage}{.333\textwidth}
\resizebox{.9\linewidth}{!}{
\begin{tikzpicture}
\begin{axis}[cycle list name=color list, ymin=-60]
\addplot table {data/all_log_r0_p7.dat};
\addplot table {data/all_log_r1_p7.dat};
\addplot table {data/all_log_r2_p7.dat};
\addplot table {data/all_log_r3_p7.dat};
\addplot table {data/all_log_r4_p7.dat};
\end{axis}
\end{tikzpicture} 
}
\end{minipage}%
\begin{minipage}{.333\textwidth}
\resizebox{.9\linewidth}{!}{
\begin{tikzpicture}
\begin{axis}[cycle list name=color list, ymin=-60]
\addplot table {data/ht_log_r0_p7.dat};
\addplot table {data/ht_log_r1_p7.dat};
\addplot table {data/ht_log_r2_p7.dat};
\addplot table {data/ht_log_r3_p7.dat};
\addplot table {data/ht_log_r4_p7.dat};
\end{axis}
\end{tikzpicture} 
}
\end{minipage}
\caption{Log weight with $p=7$ and various ranks for
$\Epr^{\mathrm{SW}}$ (left),  $\Eall^{\mathrm{SW}}$ (middle), and 
$\Eht^{\mathrm{db}}$ (right); all graphs have the same legend}
\label{fig:Ecomp}
\end{figure}

\begin{figure}
\begin{minipage}{.5\textwidth}
\resizebox{.9\linewidth}{!}{
\begin{tikzpicture}
\begin{axis}[cycle list name=color list, legend pos = outer north east, ymin = -4]
\addplot table {data/prime_unwt_r0_p7.dat};
\addplot table {data/prime_unwt_r1_p7.dat};
\addplot table {data/prime_unwt_r2_p7.dat};
\addplot table {data/prime_unwt_r3_p7.dat};
\addplot table {data/prime_unwt_r4_p7.dat};
\legend{$r= 0$, $r=1$, $r=2$, $r=3$, $r=4$}
\end{axis}
\end{tikzpicture} 
}
\end{minipage}%
\begin{minipage}{.5\textwidth}
\resizebox{.9\linewidth}{!}{
\begin{tikzpicture}
\begin{axis}[cycle list name=color list, legend pos = outer north east]
\addplot table {data/prime_sqrt_r0_p7.dat};
\addplot table {data/prime_sqrt_r1_p7.dat};
\addplot table {data/prime_sqrt_r2_p7.dat};
\addplot table {data/prime_sqrt_r3_p7.dat};
\addplot table {data/prime_sqrt_r4_p7.dat};
\legend{$r= 0$, $r=1$, $r=2$, $r=3$, $r=4$}
\end{axis}
\end{tikzpicture} 
}
\end{minipage}
\caption{Constant (left) and $\sqrt{N}$ (right) weights for $p=7$ and 
$\calE = \Epr^{\mathrm{SW}}$}
 \label{fig:un-sqrt}
\end{figure}

\begin{figure}
\begin{minipage}{.5\textwidth}
\resizebox{.9\linewidth}{!}{
\begin{tikzpicture}
\begin{axis}[cycle list name=color list, legend pos = outer north east]
\addplot table {data/prime_x_r0_p7.dat};
\addplot table {data/prime_x_r1_p7.dat};
\addplot table {data/prime_x_r2_p7.dat};
\addplot table {data/prime_x_r3_p7.dat};
\addplot table {data/prime_x_r4_p7.dat};
\legend{$r= 0$, $r=1$, $r=2$, $r=3$, $r=4$}
\end{axis}
\end{tikzpicture} 
}
\end{minipage}%
\begin{minipage}{.5\textwidth}
\resizebox{.9\linewidth}{!}{
\begin{tikzpicture}
\begin{axis}[cycle list name=color list, legend pos = outer north east]
\addplot table {data/prime_x2_r0_p7.dat};
\addplot table {data/prime_x2_r1_p7.dat};
\addplot table {data/prime_x2_r2_p7.dat};
\addplot table {data/prime_x2_r3_p7.dat};
\addplot table {data/prime_x2_r4_p7.dat};
\legend{$r= 0$, $r=1$, $r=2$, $r=3$, $r=4$}
\end{axis}
\end{tikzpicture} 
}
\end{minipage}
\caption{$N$ (left) and $N^2$ (right) weights for $p=7$ and 
$\calE = \Epr^{\mathrm{SW}}$}
 \label{fig:x-x2}
\end{figure}

These graphs also support \cref{conj2}, but due to the scale, 
the overlaid graphs in \cref{fig:Ecomp}
hide the precise behavior of the individual graphs.
See \cref{fig:rcomp-prime,fig:rcomp-all,fig:rcomp-ht} for graphs
of individual ranks with $p=7$ and $\phi(N) = \log N$
for the families $\calE = \Epr^{\mathrm{SW}}$, $\Eall^{\mathrm{SW}}$
and $\Eht^{\mathrm{db}}$, respectively.

\begin{figure}
\begin{minipage}{.333\textwidth}
\resizebox{.9\linewidth}{!}{
\begin{tikzpicture}
\begin{axis}[cycle list name=color list, legend pos = south east, ymin = 5]
\addplot table {data/prime_log_r0_p7.dat};
\legend{$r = 0$}
\end{axis}
\end{tikzpicture} 
}
\end{minipage}%
\begin{minipage}{.333\textwidth}
\resizebox{.9\linewidth}{!}{
\begin{tikzpicture}
\begin{axis}[cycle list name=color list, legend pos = south east, ymin = 1]
\addplot table {data/prime_log_r1_p7.dat};
\legend{$r = 1$}
\end{axis}
\end{tikzpicture} 
}
\end{minipage}%
\begin{minipage}{.333\textwidth}
\resizebox{.9\linewidth}{!}{
\begin{tikzpicture}
\begin{axis}[cycle list name=color list, legend pos = south east, ymin = -13]
\addplot table {data/prime_log_r2_p7.dat};
\legend{$r = 2$}
\end{axis}
\end{tikzpicture} 
}
\end{minipage}%

\begin{minipage}{.333\textwidth}
\resizebox{.9\linewidth}{!}{
\begin{tikzpicture}
\begin{axis}[cycle list name=color list, legend pos = south east, ymin = -30]
\addplot table {data/prime_log_r3_p7.dat};
\legend{$r = 3$}
\end{axis}
\end{tikzpicture} 
}
\end{minipage}%
\begin{minipage}{.333\textwidth}
\resizebox{.9\linewidth}{!}{
\begin{tikzpicture}
\begin{axis}[cycle list name=color list, legend pos = south east, ymin = -50]
\addplot table {data/prime_log_r4_p7.dat};
\legend{$r = 4$}
\end{axis}
\end{tikzpicture} 
}
\end{minipage}%
\begin{minipage}{.333\textwidth}
\resizebox{.9\linewidth}{!}{
\begin{tikzpicture}
\begin{axis}[cycle list name=color list, legend pos = south east]
\addplot table {data/prime_log_r5_p7.dat};
\legend{$r = 5$}
\end{axis}
\end{tikzpicture} 
}
\end{minipage}
\caption{Log weight for $\Epr^{\mathrm{SW}}$ with $p=7$}
\label{fig:rcomp-prime}
\end{figure}

\begin{figure}
\begin{minipage}{.333\textwidth}
\resizebox{.9\linewidth}{!}{
\begin{tikzpicture}
\begin{axis}[cycle list name=color list, legend pos = south east, ymin = 5]
\addplot table {data/all_log_r0_p7.dat};
\legend{$r = 0$}
\end{axis}
\end{tikzpicture} 
}
\end{minipage}%
\begin{minipage}{.333\textwidth}
\resizebox{.9\linewidth}{!}{
\begin{tikzpicture}
\begin{axis}[cycle list name=color list, legend pos = south east, ymin = 0]
\addplot table {data/all_log_r1_p7.dat};
\legend{$r = 1$}
\end{axis}
\end{tikzpicture} 
}
\end{minipage}%
\begin{minipage}{.333\textwidth}
\resizebox{.9\linewidth}{!}{
\begin{tikzpicture}
\begin{axis}[cycle list name=color list, legend pos = south east, ymin = -20]
\addplot table {data/all_log_r2_p7.dat};
\legend{$r = 2$}
\end{axis}
\end{tikzpicture} 
}
\end{minipage}%

\begin{minipage}{.333\textwidth}
\resizebox{.9\linewidth}{!}{
\begin{tikzpicture}
\begin{axis}[cycle list name=color list, legend pos = south east, ymin = -35]
\addplot table {data/all_log_r3_p7.dat};
\legend{$r = 3$}
\end{axis}
\end{tikzpicture} 
}
\end{minipage}%
\begin{minipage}{.333\textwidth}
\resizebox{.9\linewidth}{!}{
\begin{tikzpicture}
\begin{axis}[cycle list name=color list, legend pos = south east, ymin = -60]
\addplot table {data/all_log_r4_p7.dat};
\legend{$r = 4$}
\end{axis}
\end{tikzpicture} 
}
\end{minipage}%
\begin{minipage}{.333\textwidth}
{}
\end{minipage}
\caption{Log weight for $\Eall^{\mathrm{SW}}$ with $p=7$}
\label{fig:rcomp-all}
\end{figure}

\begin{figure}
\begin{minipage}{.333\textwidth}
\resizebox{.9\linewidth}{!}{
\begin{tikzpicture}
\begin{axis}[cycle list name=color list, legend pos = south east, ymin = 5, ymax = 7]
\addplot table {data/ht_log_r0_p7.dat};
\legend{$r = 0$}
\end{axis}
\end{tikzpicture} 
}
\end{minipage}%
\begin{minipage}{.333\textwidth}
\resizebox{.9\linewidth}{!}{
\begin{tikzpicture}
\begin{axis}[cycle list name=color list, legend pos = south east, ymin =0, ymax = 3]
\addplot table {data/ht_log_r1_p7.dat};
\legend{$r = 1$}
\end{axis}
\end{tikzpicture} 
}
\end{minipage}%
\begin{minipage}{.333\textwidth}
\resizebox{.9\linewidth}{!}{
\begin{tikzpicture}
\begin{axis}[cycle list name=color list, legend pos = south east, ymin=-14, ymax = -10]
\addplot table {data/ht_log_r2_p7.dat};
\legend{$r = 2$}
\end{axis}
\end{tikzpicture} 
}
\end{minipage}%

\begin{minipage}{.333\textwidth}
\resizebox{.9\linewidth}{!}{
\begin{tikzpicture}
\begin{axis}[cycle list name=color list, legend pos = south east, ymin = -30, ymax = -23]
\addplot table {data/ht_log_r3_p7.dat};
\legend{$r = 3$}
\end{axis}
\end{tikzpicture} 
}
\end{minipage}%
\begin{minipage}{.333\textwidth}
\resizebox{.9\linewidth}{!}{
\begin{tikzpicture}
\begin{axis}[cycle list name=color list, legend pos = south east, ymin = -50, ymax = -40]
\addplot table {data/ht_log_r4_p7.dat};
\legend{$r = 4$}
\end{axis}
\end{tikzpicture} 
}
\end{minipage}%
\begin{minipage}{.333\textwidth}
\resizebox{.9\linewidth}{!}{
\begin{tikzpicture}
\begin{axis}[cycle list name=color list, legend pos = south east, ymin = -80, ymax = -60]
\addplot table {data/ht_log_r5_p7.dat};
\legend{$r = 5$}
\end{axis}
\end{tikzpicture} 
}
\end{minipage}
\caption{Log weight for $\Eht^{\mathrm{d}}$ with $p=7$}
\label{fig:rcomp-ht}
\end{figure}

Note that for the prime conductor and height databases with log weighting, 
for each $r$ the graph appears to be eventually relatively flat or possibly
tending very slowly away from 0.  The case of $\calE = \Eall^{\mathrm{SW}}$
is different however: except for $r=1$ where the graph appears to be slowing
increasing away from 0, but for all other ranks the graphs appear to bend towards
0.  While it may be that the behavior is actually different for the family
$\Eall$, we suspect this difference is more likely due to factors such as
the Stein--Watkins all conductor database being rather incomplete (and perhaps 
giving a biased sample of $\Eall$) and only going up to conductor $10^8$.
However, even just restricting to the prime conductor and height databases,
we see that for some ranks the log weighted averages appear to flatten out and
for some ranks (notably $r=2$ and $r=3$) they appear to be slowly tending away
from 0.  (We have also included $r=5$ graphs for the prime conductor and height 
databases, but there is perhaps not enough data to read too much into these graphs.)

To try to get a more precise sense of the order of bias, in
\cref{fig:r0comp,fig:r1comp,fig:r2comp} 
we present several graphs for ranks 0--2
for each of our 3 databases.  (Analogous graphs for ranks 3 and 4 look
similar to the rank 2 graphs, and we omit them.)
Namely, we graph $\calA^\calE(p, X; \phi)$
for $p = 7$ and $p = 11$ with the 3 weight functions $\phi(N) = \log \log N$,
$\phi(N) = \log N$ and $\phi(N) = (\log N)^2$.
In all cases except for
$\calE = \Eall^{\mathrm{SW}}$ with $r=1$ and $\phi(N) = \log \log N$, 
we see that the graphs are
tending toward 0 for the log log weight, and tending toward $\pm \infty$
for the $\log^2$ weight.

\begin{figure}
\begin{minipage}{.333\textwidth}
\resizebox{.9\linewidth}{!}{
\begin{tikzpicture}
\begin{axis}[cycle list name=color list, legend pos = south east, ymin=0.5,ymax=1.5]
\addplot table {data/prime_loglog_r0_p7.dat};
\addplot table {data/prime_loglog_r0_p11.dat};
\legend{$p = 7$, $p = 11$}
\end{axis}
\end{tikzpicture} 
}
\end{minipage}%
\begin{minipage}{.333\textwidth}
\resizebox{.9\linewidth}{!}{
\begin{tikzpicture}
\begin{axis}[cycle list name=color list, legend pos = south east, ymin=0.5,ymax=1.5]
\addplot table {data/all_loglog_r0_p7.dat};
\addplot table {data/all_loglog_r0_p11.dat};
\legend{$p = 7$, $p = 11$}
\end{axis}
\end{tikzpicture} 
}
\end{minipage}%
\begin{minipage}{.333\textwidth}
\resizebox{.9\linewidth}{!}{
\begin{tikzpicture}
\begin{axis}[cycle list name=color list, legend pos = south east, ymin=0.5,ymax=1.5]
\addplot table {data/ht_loglog_r0_p7.dat};
\addplot table {data/ht_loglog_r0_p11.dat};
\legend{$p = 7$, $p = 11$}
\end{axis}
\end{tikzpicture} 
}
\end{minipage}

\begin{minipage}{.333\textwidth}
\resizebox{.9\linewidth}{!}{
\begin{tikzpicture}
\begin{axis}[cycle list name=color list, legend pos = south east, ymin=5, ymax = 7.5]
\addplot table {data/prime_log_r0_p7.dat};
\addplot table {data/prime_log_r0_p11.dat};
\legend{$p = 7$, $p = 11$}
\end{axis}
\end{tikzpicture} 
}
\end{minipage}%
\begin{minipage}{.333\textwidth}
\resizebox{.9\linewidth}{!}{
\begin{tikzpicture}
\begin{axis}[cycle list name=color list, legend pos = south east, ymin=5, ymax = 7.5]
\addplot table {data/all_log_r0_p7.dat};
\addplot table {data/all_log_r0_p11.dat};
\legend{$p = 7$, $p = 11$}
\end{axis}
\end{tikzpicture} 
}
\end{minipage}%
\begin{minipage}{.333\textwidth}
\resizebox{.9\linewidth}{!}{
\begin{tikzpicture}
\begin{axis}[cycle list name=color list, legend pos = south east, ymin=5, ymax = 7.5]
\addplot table {data/ht_log_r0_p7.dat};
\addplot table {data/ht_log_r0_p11.dat};
\legend{$p = 7$, $p = 11$}
\end{axis}
\end{tikzpicture} 
}
\end{minipage}%

\begin{minipage}{.333\textwidth}
\resizebox{.9\linewidth}{!}{
\begin{tikzpicture}
\begin{axis}[cycle list name=color list, legend pos = south east, ymin = 50, ymax = 170]
\addplot table {data/prime_log2_r0_p7.dat};
\addplot table {data/prime_log2_r0_p11.dat};
\legend{$p = 7$, $p = 11$}
\end{axis}
\end{tikzpicture} 
}
\end{minipage}%
\begin{minipage}{.333\textwidth}
\resizebox{.9\linewidth}{!}{
\begin{tikzpicture}
\begin{axis}[cycle list name=color list, legend pos = south east, ymin = 50, ymax = 170]
\addplot table {data/all_log2_r0_p7.dat};
\addplot table {data/all_log2_r0_p11.dat};
\legend{$p = 7$, $p = 11$}
\end{axis}
\end{tikzpicture} 
}
\end{minipage}%
\begin{minipage}{.333\textwidth}
\resizebox{.9\linewidth}{!}{
\begin{tikzpicture}
\begin{axis}[cycle list name=color list, legend pos = south east, ymin = 50, ymax = 170]
\addplot table {data/ht_log2_r0_p7.dat};
\addplot table {data/ht_log2_r0_p11.dat};
\legend{$p = 7$, $p = 11$}
\end{axis}
\end{tikzpicture} 
}
\end{minipage}%

\caption{Rank 0 graphs with log log (top), log (middle) and $\log^2$ (bottom) weights for $\Epr^{\mathrm{SW}}$ (left) and  $\Eall^{\mathrm{SW}}$ (middle), and 
$\Eht^{\mathrm{db}}$ (right)}
\label{fig:r0comp}
\end{figure}

\begin{figure}
\begin{minipage}{.333\textwidth}
\resizebox{.9\linewidth}{!}{
\begin{tikzpicture}
\begin{axis}[cycle list name=color list, legend pos = south east, ymin=0,ymax=0.4]
\addplot table {data/prime_loglog_r1_p7.dat};
\addplot table {data/prime_loglog_r1_p11.dat};
\legend{$p = 7$, $p = 11$}
\end{axis}
\end{tikzpicture} 
}
\end{minipage}%
\begin{minipage}{.333\textwidth}
\resizebox{.9\linewidth}{!}{
\begin{tikzpicture}
\begin{axis}[cycle list name=color list, legend pos = south east, ymin=0,ymax=0.4]
\addplot table {data/all_loglog_r1_p7.dat};
\addplot table {data/all_loglog_r1_p11.dat};
\legend{$p = 7$, $p = 11$}
\end{axis}
\end{tikzpicture} 
}
\end{minipage}%
\begin{minipage}{.333\textwidth}
\resizebox{.9\linewidth}{!}{
\begin{tikzpicture}
\begin{axis}[cycle list name=color list, legend pos = south east, ymin=0,ymax=0.4]
\addplot table {data/ht_loglog_r1_p7.dat};
\addplot table {data/ht_loglog_r1_p11.dat};
\legend{$p = 7$, $p = 11$}
\end{axis}
\end{tikzpicture} 
}
\end{minipage}

\begin{minipage}{.333\textwidth}
\resizebox{.9\linewidth}{!}{
\begin{tikzpicture}
\begin{axis}[cycle list name=color list, legend pos = south east, ymin=0.5, ymax = 2.5]
\addplot table {data/prime_log_r1_p7.dat};
\addplot table {data/prime_log_r1_p11.dat};
\legend{$p = 7$, $p = 11$}
\end{axis}
\end{tikzpicture} 
}
\end{minipage}%
\begin{minipage}{.333\textwidth}
\resizebox{.9\linewidth}{!}{
\begin{tikzpicture}
\begin{axis}[cycle list name=color list, legend pos = south east, ymin=0.5, ymax = 2.5]
\addplot table {data/all_log_r1_p7.dat};
\addplot table {data/all_log_r1_p11.dat};
\legend{$p = 7$, $p = 11$}
\end{axis}
\end{tikzpicture} 
}
\end{minipage}%
\begin{minipage}{.333\textwidth}
\resizebox{.9\linewidth}{!}{
\begin{tikzpicture}
\begin{axis}[cycle list name=color list, legend pos = south east, ymin = 0.5, ymax = 2.5]
\addplot table {data/ht_log_r1_p7.dat};
\addplot table {data/ht_log_r1_p11.dat};
\legend{$p = 7$, $p = 11$}
\end{axis}
\end{tikzpicture} 
}
\end{minipage}%

\begin{minipage}{.333\textwidth}
\resizebox{.9\linewidth}{!}{
\begin{tikzpicture}
\begin{axis}[cycle list name=color list, legend pos = south east, ymin = 10, ymax = 50]
\addplot table {data/prime_log2_r1_p7.dat};
\addplot table {data/prime_log2_r1_p11.dat};
\legend{$p = 7$, $p = 11$}
\end{axis}
\end{tikzpicture} 
}
\end{minipage}%
\begin{minipage}{.333\textwidth}
\resizebox{.9\linewidth}{!}{
\begin{tikzpicture}
\begin{axis}[cycle list name=color list, legend pos = south east, ymin = 10, ymax = 50]
\addplot table {data/all_log2_r1_p7.dat};
\addplot table {data/all_log2_r1_p11.dat};
\legend{$p = 7$, $p = 11$}
\end{axis}
\end{tikzpicture} 
}
\end{minipage}%
\begin{minipage}{.333\textwidth}
\resizebox{.9\linewidth}{!}{
\begin{tikzpicture}
\begin{axis}[cycle list name=color list, legend pos = south east, ymin = 10, ymax = 50]
\addplot table {data/ht_log2_r1_p7.dat};
\addplot table {data/ht_log2_r1_p11.dat};
\legend{$p = 7$, $p = 11$}
\end{axis}
\end{tikzpicture} 
}
\end{minipage}%

\caption{Rank 1 graphs with log log (top), log (middle) and $\log^2$ (bottom) weights for $\Epr^{\mathrm{SW}}$ (left) and  $\Eall^{\mathrm{SW}}$ (middle), and 
$\Eht^{\mathrm{db}}$ (right)}
\label{fig:r1comp}
\end{figure}

\begin{figure}
\begin{minipage}{.333\textwidth}
\resizebox{.9\linewidth}{!}{
\begin{tikzpicture}
\begin{axis}[cycle list name=color list, legend pos = south east, ymin = -4, ymax = -1]
\addplot table {data/prime_loglog_r2_p7.dat};
\addplot table {data/prime_loglog_r2_p11.dat};
\legend{$p = 7$, $p = 11$}
\end{axis}
\end{tikzpicture} 
}
\end{minipage}%
\begin{minipage}{.333\textwidth}
\resizebox{.9\linewidth}{!}{
\begin{tikzpicture}
\begin{axis}[cycle list name=color list, legend pos = south east, ymin = -4, ymax = -1]
\addplot table {data/all_loglog_r2_p7.dat};
\addplot table {data/all_loglog_r2_p11.dat};
\legend{$p = 7$, $p = 11$}
\end{axis}
\end{tikzpicture} 
}
\end{minipage}%
\begin{minipage}{.333\textwidth}
\resizebox{.9\linewidth}{!}{
\begin{tikzpicture}
\begin{axis}[cycle list name=color list, legend pos = south east, ymin = -4, ymax = -1]
\addplot table {data/ht_loglog_r2_p7.dat};
\addplot table {data/ht_loglog_r2_p11.dat};
\legend{$p = 7$, $p = 11$}
\end{axis}
\end{tikzpicture} 
}
\end{minipage}

\begin{minipage}{.333\textwidth}
\resizebox{.9\linewidth}{!}{
\begin{tikzpicture}
\begin{axis}[cycle list name=color list, legend pos = south east, ymin = -17, ymax = -8]
\addplot table {data/prime_log_r2_p7.dat};
\addplot table {data/prime_log_r2_p11.dat};
\legend{$p = 7$, $p = 11$}
\end{axis}
\end{tikzpicture} 
}
\end{minipage}%
\begin{minipage}{.333\textwidth}
\resizebox{.9\linewidth}{!}{
\begin{tikzpicture}
\begin{axis}[cycle list name=color list, legend pos = south east, ymin = -17, ymax = -8]
\addplot table {data/all_log_r2_p7.dat};
\addplot table {data/all_log_r2_p11.dat};
\legend{$p = 7$, $p = 11$}
\end{axis}
\end{tikzpicture} 
}
\end{minipage}%
\begin{minipage}{.333\textwidth}
\resizebox{.9\linewidth}{!}{
\begin{tikzpicture}
\begin{axis}[cycle list name=color list, legend pos = south east, ymin = -17, ymax = -8]
\addplot table {data/ht_log_r2_p7.dat};
\addplot table {data/ht_log_r2_p11.dat};
\legend{$p = 7$, $p = 11$}
\end{axis}
\end{tikzpicture} 
}
\end{minipage}%

\begin{minipage}{.333\textwidth}
\resizebox{.9\linewidth}{!}{
\begin{tikzpicture}
\begin{axis}[cycle list name=color list, legend pos = south east, ymin = -300, ymax = -150]
\addplot table {data/prime_log2_r2_p7.dat};
\addplot table {data/prime_log2_r2_p11.dat};
\legend{$p = 7$, $p = 11$}
\end{axis}
\end{tikzpicture} 
}
\end{minipage}%
\begin{minipage}{.333\textwidth}
\resizebox{.9\linewidth}{!}{
\begin{tikzpicture}
\begin{axis}[cycle list name=color list, legend pos = south east, ymin = -300, ymax = -150]
\addplot table {data/all_log2_r2_p7.dat};
\addplot table {data/all_log2_r2_p11.dat};
\legend{$p = 7$, $p = 11$}
\end{axis}
\end{tikzpicture} 
}
\end{minipage}%
\begin{minipage}{.333\textwidth}
\resizebox{.9\linewidth}{!}{
\begin{tikzpicture}
\begin{axis}[cycle list name=color list, legend pos = north east, ymin = -300, ymax = -150]
\addplot table {data/ht_log2_r2_p7.dat};
\addplot table {data/ht_log2_r2_p11.dat};
\legend{$p = 7$, $p = 11$}
\end{axis}
\end{tikzpicture} 
}
\end{minipage}%

\caption{Rank 2 graphs with  log log (top), log (middle) and $\log^2$ (bottom) weights for $\Epr^{\mathrm{SW}}$ (left) and  $\Eall^{\mathrm{SW}}$ (middle), and 
$\Eht^{\mathrm{db}}$ (right)}
\label{fig:r2comp}
\end{figure}

\ifgraphs
\begin{figure}
\begin{minipage}{.333\textwidth}
\resizebox{.9\linewidth}{!}{
\begin{tikzpicture}
\begin{axis}[cycle list name=color list, legend pos = south east, ymax = -2, ymin = -8]
\addplot table {data/prime_loglog_r3_p7.dat};
\addplot table {data/prime_loglog_r3_p11.dat};
\legend{$p = 7$, $p = 11$}
\end{axis}
\end{tikzpicture} 
}
\end{minipage}%
\begin{minipage}{.333\textwidth}
\resizebox{.9\linewidth}{!}{
\begin{tikzpicture}
\begin{axis}[cycle list name=color list, legend pos = south east, ymax = -2, ymin = -8]
\addplot table {data/all_loglog_r3_p7.dat};
\addplot table {data/all_loglog_r3_p11.dat};
\legend{$p = 7$, $p = 11$}
\end{axis}
\end{tikzpicture} 
}
\end{minipage}%
\begin{minipage}{.333\textwidth}
\resizebox{.9\linewidth}{!}{
\begin{tikzpicture}
\begin{axis}[cycle list name=color list, legend pos = south east, ymax = -2, ymin = -8]
\addplot table {data/ht_loglog_r3_p7.dat};
\addplot table {data/ht_loglog_r3_p11.dat};
\legend{$p = 7$, $p = 11$}
\end{axis}
\end{tikzpicture} 
}
\end{minipage}

\begin{minipage}{.333\textwidth}
\resizebox{.9\linewidth}{!}{
\begin{tikzpicture}
\begin{axis}[cycle list name=color list, legend pos = south east, ymax = -20, ymin = -40]
\addplot table {data/prime_log_r3_p7.dat};
\addplot table {data/prime_log_r3_p11.dat};
\legend{$p = 7$, $p = 11$}
\end{axis}
\end{tikzpicture} 
}
\end{minipage}%
\begin{minipage}{.333\textwidth}
\resizebox{.9\linewidth}{!}{
\begin{tikzpicture}
\begin{axis}[cycle list name=color list, legend pos = north east, ymax = -20, ymin = -40]
\addplot table {data/all_log_r3_p7.dat};
\addplot table {data/all_log_r3_p11.dat};
\legend{$p = 7$, $p = 11$}
\end{axis}
\end{tikzpicture} 
}
\end{minipage}%
\begin{minipage}{.333\textwidth}
\resizebox{.9\linewidth}{!}{
\begin{tikzpicture}
\begin{axis}[cycle list name=color list, legend pos = south east, ymax = -20, ymin = -40]
\addplot table {data/ht_log_r3_p7.dat};
\addplot table {data/ht_log_r3_p11.dat};
\legend{$p = 7$, $p = 11$}
\end{axis}
\end{tikzpicture} 
}
\end{minipage}%

\begin{minipage}{.333\textwidth}
\resizebox{.9\linewidth}{!}{
\begin{tikzpicture}
\begin{axis}[cycle list name=color list, legend pos = south east, ymax = -400, ymin = -800]
\addplot table {data/prime_log2_r3_p7.dat};
\addplot table {data/prime_log2_r3_p11.dat};
\legend{$p = 7$, $p = 11$}
\end{axis}
\end{tikzpicture} 
}
\end{minipage}%
\begin{minipage}{.333\textwidth}
\resizebox{.9\linewidth}{!}{
\begin{tikzpicture}
\begin{axis}[cycle list name=color list, legend pos = south east, ymax = -400, ymin = -800]
\addplot table {data/all_log2_r3_p7.dat};
\addplot table {data/all_log2_r3_p11.dat};
\legend{$p = 7$, $p = 11$}
\end{axis}
\end{tikzpicture} 
}
\end{minipage}%
\begin{minipage}{.333\textwidth}
\resizebox{.9\linewidth}{!}{
\begin{tikzpicture}
\begin{axis}[cycle list name=color list, legend pos = north east, ymax = -400, ymin = -800]
\addplot table {data/ht_log2_r3_p7.dat};
\addplot table {data/ht_log2_r3_p11.dat};
\legend{$p = 7$, $p = 11$}
\end{axis}
\end{tikzpicture} 
}
\end{minipage}%

\caption{Rank 3 graphs with log log (top), log (middle) and $\log^2$ (bottom) weights for $\Epr^{\mathrm{SW}}$ (left) and  $\Eall^{\mathrm{SW}}$ (middle), and 
$\Eht^{\mathrm{db}}$ (right)}
\label{fig:r3comp}
\end{figure}

\begin{figure}
\begin{minipage}{.333\textwidth}
\resizebox{.9\linewidth}{!}{
\begin{tikzpicture}
\begin{axis}[cycle list name=color list, legend pos = south east, ymax = -5, ymin = -13]
\addplot table {data/prime_loglog_r4_p7.dat};
\addplot table {data/prime_loglog_r4_p11.dat};
\legend{$p = 7$, $p = 11$}
\end{axis}
\end{tikzpicture} 
}
\end{minipage}%
\begin{minipage}{.333\textwidth}
\resizebox{.9\linewidth}{!}{
\begin{tikzpicture}
\begin{axis}[cycle list name=color list, legend pos = south east, ymax = -5, ymin = -13]
\addplot table {data/all_loglog_r4_p7.dat};
\addplot table {data/all_loglog_r4_p11.dat};
\legend{$p = 7$, $p = 11$}
\end{axis}
\end{tikzpicture} 
}
\end{minipage}%
\begin{minipage}{.333\textwidth}
\resizebox{.9\linewidth}{!}{
\begin{tikzpicture}
\begin{axis}[cycle list name=color list, legend pos = south east, ymax = -5, ymin = -13]
\addplot table {data/ht_loglog_r4_p7.dat};
\addplot table {data/ht_loglog_r4_p11.dat};
\legend{$p = 7$, $p = 11$}
\end{axis}
\end{tikzpicture} 
}
\end{minipage}

\begin{minipage}{.333\textwidth}
\resizebox{.9\linewidth}{!}{
\begin{tikzpicture}
\begin{axis}[cycle list name=color list, legend pos = south east, ymax = -40, ymin = -70]
\addplot table {data/prime_log_r4_p7.dat};
\addplot table {data/prime_log_r4_p11.dat};
\legend{$p = 7$, $p = 11$}
\end{axis}
\end{tikzpicture} 
}
\end{minipage}%
\begin{minipage}{.333\textwidth}
\resizebox{.9\linewidth}{!}{
\begin{tikzpicture}
\begin{axis}[cycle list name=color list, legend pos = north east, ymax = -40, ymin = -70]
\addplot table {data/all_log_r4_p7.dat};
\addplot table {data/all_log_r4_p11.dat};
\legend{$p = 7$, $p = 11$}
\end{axis}
\end{tikzpicture} 
}
\end{minipage}%
\begin{minipage}{.333\textwidth}
\resizebox{.9\linewidth}{!}{
\begin{tikzpicture}
\begin{axis}[cycle list name=color list, legend pos = south east, ymax = -40, ymin = -70]
\addplot table {data/ht_log_r4_p7.dat};
\addplot table {data/ht_log_r4_p11.dat};
\legend{$p = 7$, $p = 11$}
\end{axis}
\end{tikzpicture} 
}
\end{minipage}%

\begin{minipage}{.333\textwidth}
\resizebox{.9\linewidth}{!}{
\begin{tikzpicture}
\begin{axis}[cycle list name=color list, legend pos = south east, ymax = -750, ymin = -1300]
\addplot table {data/prime_log2_r4_p7.dat};
\addplot table {data/prime_log2_r4_p11.dat};
\legend{$p = 7$, $p = 11$}
\end{axis}
\end{tikzpicture} 
}
\end{minipage}%
\begin{minipage}{.333\textwidth}
\resizebox{.9\linewidth}{!}{
\begin{tikzpicture}
\begin{axis}[cycle list name=color list, legend pos = south east, ymax = -750, ymin = -1300]
\addplot table {data/all_log2_r4_p7.dat};
\addplot table {data/all_log2_r4_p11.dat};
\legend{$p = 7$, $p = 11$}
\end{axis}
\end{tikzpicture} 
}
\end{minipage}%
\begin{minipage}{.333\textwidth}
\resizebox{.9\linewidth}{!}{
\begin{tikzpicture}
\begin{axis}[cycle list name=color list, legend pos = north east, ymax = -750, ymin = -1300]
\addplot table {data/ht_log2_r4_p7.dat};
\addplot table {data/ht_log2_r4_p11.dat};
\legend{$p = 7$, $p = 11$}
\end{axis}
\end{tikzpicture} 
}
\end{minipage}%

\caption{Rank 4 graphs with log log (top), log (middle) and $\log^2$ (bottom) weights for $\Epr^{\mathrm{SW}}$ (left) and  $\Eall^{\mathrm{SW}}$ (middle), and 
$\Eht^{\mathrm{db}}$ (right)}
\label{fig:r4comp}
\end{figure}
\fi

When $r=1$, there is another anomaly besides the log log weight 
for $\calE = \Eall^{\mathrm{SW}}$.  Namely, for 
$\calE = \Epr^{\mathrm{SW}}$ the log weighted graph appears to be decreasing
for $p=11$ and flat or slightly increasing for $p=7$.  
Again we suspect these anomalies are due to limited data---the prime conductor
database contains significantly fewer curves than the other databases.

Supporting this idea, in \cref{fig:prange-r2} we present rank 2 graphs for 
$\calE = \Eall^{\mathrm{SW}}$ and log weight for a variety of primes,
and we see that the graphs have the same shape for each of our choices
of primes.  (Of course we expect that the weighted averages tend to have
larger absolute value for larger $p$ because the Hecke bound grows with $p$.)
If one plots the same graphs for rank 0 or 1, we again see the shape
is essentially independent of $p$ (the case $p=79$ case with $r=0$ is an exception).
Now if we look at rank 3 graphs for some of our larger values of $p$,
the shape seems to depend on $p$.  We plotted these graphs together with
the corresponding graphs for the height database in \cref{fig:prange-r3}.
In the case of the height database, where there are many more rank 3 curves,
we again see the graphs have essentially the same shape for each $p$.

Thus we expect that the general behavior of weighted averages is independent
of $p$, but it may require a very large amount of data to numerically see the asymptotic
behavior, especially for larger $p$ and $\phi$ having logarithmic growth.

In summary, we find the data quite suggestive that, for each of our 3 families,
the weighted averages
 $\calA_r^\calE(p, X; \phi)$ tend to 0 if $\phi(N) \ll \log \log N$ and tend to $+\infty$
(resp.\ $-\infty$) if $\phi(N) \gg (\log N)^2$ and $r \le 1$ (resp.\ $r \ge 2$).
When $\phi(N) = \log N$, some of the graphs appear very flat, and some do not.
Because of the variation of these graphs, the precise asymptotic behavior is not 
apparent with the databases currently available to us, but it seems plausible
that the value of $\delta$ in \cref{conj2} may be 1.


\begin{figure}
\begin{minipage}{.5\textwidth}
\resizebox{.9\linewidth}{!}{
\begin{tikzpicture}
\begin{axis}[cycle list name=color list, legend pos = outer north east, ymax = -5, ymin = -20]
\addplot table {data/all_log_r2_p2.dat};
\addplot table {data/all_log_r2_p3.dat};
\addplot table {data/all_log_r2_p5.dat};
\addplot table {data/all_log_r2_p7.dat};
\addplot table {data/all_log_r2_p11.dat};
\addplot table {data/all_log_r2_p17.dat};
\legend{$p = 2$, $p=3$, $p = 5$, $p = 7$, $p = 11$, $p=17$}
\end{axis}
\end{tikzpicture} 
}
\end{minipage}%
\begin{minipage}{.5\textwidth}
\resizebox{.9\linewidth}{!}{
\begin{tikzpicture}
\begin{axis}[cycle list name=color list, legend pos = outer north east,ymax=-15,ymin=-30]
\addplot table {data/all_log_r2_p43.dat};
\addplot table {data/all_log_r2_p79.dat};
\addplot table {data/all_log_r2_p101.dat};
\addplot table {data/all_log_r2_p197.dat};
\addplot table {data/all_log_r2_p223.dat};
\addplot table {data/all_log_r2_p271.dat};
\legend{$p = 43$, $p=79$, $p = 101$, $p = 197$, $p = 223$, $p=271$}
\end{axis}
\end{tikzpicture} 
}
\end{minipage}
\caption{Log weight with $r=2$ and various $p$
for $\Eall^{\mathrm{SW}}$}
\label{fig:prange-r2}
\end{figure}

\begin{figure}
\begin{minipage}{.5\textwidth}
\resizebox{.9\linewidth}{!}{
\begin{tikzpicture}
\begin{axis}[cycle list name=color list, legend pos = outer north east, ymax = -35, ymin = -65]
\addplot table {data/all_log_r3_p43.dat};
\addplot table {data/all_log_r3_p79.dat};
\addplot table {data/all_log_r3_p101.dat};
\addplot table {data/all_log_r3_p197.dat};
\addplot table {data/all_log_r3_p223.dat};
\addplot table {data/all_log_r3_p271.dat};
\legend{$p = 43$, $p=79$, $p = 101$, $p = 197$, $p = 223$, $p=271$}
\end{axis}
\end{tikzpicture} 
}
\end{minipage}%
\begin{minipage}{.5\textwidth}
\resizebox{.9\linewidth}{!}{
\begin{tikzpicture}
\begin{axis}[cycle list name=color list, legend pos = outer north east, ymax = -35, ymin = -65]
\addplot table {data/ht_log_r3_p43.dat};
\addplot table {data/ht_log_r3_p79.dat};
\addplot table {data/ht_log_r3_p101.dat};
\addplot table {data/ht_log_r3_p197.dat};
\addplot table {data/ht_log_r3_p223.dat};
\addplot table {data/ht_log_r3_p271.dat};
\legend{$p = 43$, $p=79$, $p = 101$, $p = 197$, $p = 223$, $p=271$}
\end{axis}
\end{tikzpicture} 
}
\end{minipage}
\caption{Log weight with $r=3$ and various $p$
for $\Eall^{\mathrm{SW}}$ (left) and $\Eht^{\mathrm{db}}$ (right)}
\label{fig:prange-r3}
\end{figure}


\subsection*{Acknowledgements}
This work was supported by a grant from the Simons Foundation (512927, KM).
Some of the computing for this project was performed at the OU Supercomputing Center 
for Education \& Research (OSCER) at the University of Oklahoma (OU). 


\end{document}